\documentclass[11pt,a4paper]{article}
\usepackage[margin=1.1cm]{geometry}
\usepackage{amsmath,amsthm,amsfonts,amssymb,amscd,cite,graphicx,xcolor}

\usepackage{titlesec}
\titleformat{\section}
{\normalfont\fontsize{12}{15}\bfseries}{\thesection}{1em.}{}

\newtheorem{definition}{Definition}[section]
\newtheorem{theorem}{Theorem}[section]

\allowdisplaybreaks[4]

\let\oldbibliography\thebibliography
\renewcommand{\thebibliography}[1]{%
  \oldbibliography{#1}%
  \setlength{\itemsep}{-2pt}%
}

\begin{document}

\baselineskip=0.20in

\makebox[\textwidth]{%
\hglue-15pt
\begin{minipage}{0.6cm}	
%\vfill
\vskip9pt
\end{minipage} \vspace{-\parskip}
\begin{minipage}[t]{6cm}

\end{minipage}
\hfill
\begin{minipage}[t]{6.5cm}
%\normalsize {\it Discrete Math. Lett.}  {\bf X} (202X) XX--XX
\end{minipage}}
\vskip36pt

\noindent\centering
{\large \bf Higher Order $r$-Dowling polynomials}\\

\noindent
Funani Sinethemba$^{1,}$, Ndiweni Odilo$^{1}$, Nkonkobe Sithembele$^{2,3}$ \\

\noindent
\footnotesize $^1${\it Department of pure and applied Mathematics, University of Fort Hare, Alice, 5700, South Africa}\\
\noindent
$^2${\it Department of Mathematical Sciences, Sol Plaatje University, Kimberley, 8301, South Africa \/} \\
\noindent
 $^3${\it School of Mathematics, University of Witwatersrand, 2050 Wits, Johannesburg, South Africa \/} \\

\setcounter{page}{1} \thispagestyle{empty}

\normalsize

 \begin{abstract}\normalsize
 \noindent
 Given an ordered set partition, when one insert a number of bars in-between the blocks of the ordered set partition the result is a barred preferential arrangement. 
  In this study, using the notion of barred preferential arrangements we propose a combinatorial interpretation of a type of generalized Bell polynomials. We also define a new higher order generalization of the $r$-Dowling polynomials. We discuss its degenerate and non degenerate versions. Using the notion of barred preferential arrangements we provide a combinatorial interpretation of these higher order  $r$-Dowling polynomials. Furthermore, we prove several combinatorial identities on these polynomials. We also provide some integral representations of these polynomials, and provide some of their asymptotic results. We also show several closed form expressions demonstrating how these higher order $r$-Dowling polynomials may be expressed in terms of Bell polynomials.
 \\[2mm]
 \end{abstract}

\baselineskip=10mm

\section{Introduction}\normalsize
The Stirling numbers of the second kind $S(n,k)$ may be defined in the following way (see for instance \cite{Comtet1974})

\begin{equation}\label{equation:1}
	\sum\limits_{n\geq k}S(n,k)\frac{z^n}{n!}=\frac{(e^z-1)^k}{k!}.
\end{equation}

These numbers are well known to give the number of partitions of the set $[n]=\{1,2,3,\ldots,n\}$ into $k$ non-empty subsets. A closely related concept to the Stirling numbers of the second kind is that of barred preferential arrangements.  A barred preferential arrangement of $[n]$ is a  type of preferential arrangement (ordered set partition) where a number of  identical bars are inserted between the linearly ordered blocks of a preferential arrangement \cite{Albhach}. Given $m$ bars, the $m$ bars induce $m+1$ sections onto which blocks of an ordered partition may be arranged.  The following examples are barred preferential arrangements of the sets $[5]$, and $[6]$ respectively, 
\begin{enumerate}
	\item [(a)] $1| 23 |\:\:|4\quad5$
	\item [(b)] $12|4\quad3\quad67|\:\:|5$
\end{enumerate}
The barred preferential arrangement in (a) has three bars, hence four sections. The first section, the section to the left of the first bar from left to right, has the singleton $\{1\}$. The second section of the barred preferential arrangement has the block $\{2,3\}$. The third section is empty, and the fourth section has the blocks $\{4\}$, and $\{5\}$. The first section of the barred preferential arrangement in (b) has the block $\{1,2\}$. The second section has the three blocks $\{4\}$, $\{3\}$, and $\{6,7\}$. The third section is empty, and the fourth sections has the singleton block $\{5\}$. We denote by $P_{n,l}$ the number of barred preferential arrangements of $[n]$ having $l$ bars. Barred preferential Arrangements do  appear in literature, with early mentions found in \cite{Albhach, Pippenger}. Barred preferential arrangements have the closed form\cite{Albhach}
\begin{equation}
	P_{n,l}=\sum\limits_{k=0}^nS(n,k)k!\binom{l+k-1}{k}.
\end{equation} The generating function for the number of barred preferential arrangements is\cite{Albhach}  
\begin{equation}
	P_l (z) = \sum_{n=0}^{\infty} P_{n,l}\frac{z^n}{n!} = \frac{1}{(2-e^z)^{l+1}}.
\end{equation}
Barred preferential arrangements have extensively been studied in \cite{adell2023unified,a2023generalized,benyi2024combinatorial,corcino2022second,nkonkobe2017study,Nkonkobe et al}. 

Researchers have been also interested in extending stirling numbers, exploring both their combinatorial and number theoretic properties of their extensions. Let $e^r_\alpha(z)=\begin{pmatrix}
1+\alpha z
\end{pmatrix}^\frac{r}{\alpha}$. A unified generalization of extensions of stirling numbers $S(n,k,\alpha, \beta, \gamma)$ defined in  \cite{Hsu&Shiue1998} in the following way
\begin{equation}
		\sum_{n=0}^{\infty} S(n,k,\alpha, \beta,\gamma) \frac{z^n}{n!} = \frac{e^r_\alpha(z)}{k!} \biggl[\frac{e^\beta_\alpha(z) -1}{\beta}\biggr]^k , 
\end{equation}

 where the parameters $\alpha, \beta, \gamma$ are real or complex and $(\alpha, \beta, \gamma)\neq (0,0,0)$. This unified generalization also satisfies the following recurrence relation (see \cite{Hsu&Shiue1998}) 
\begin{equation}
	S(n+1,k,\alpha, \beta, r) = S(n,k-1, \alpha ,\beta, \gamma) + (k\beta -n\alpha +r)S(n,k,\alpha, \beta, r) ,
\end{equation}
where $n\geq k$ and $k\geq 1$. The generalization  $S(n,k,\alpha, \beta,\gamma)$ is a unified generalization of several extensions of the stirling numbers, below we list a couple of them.  
\begin{table}[!h] 
	\begin{tabular}{|c|c|} \hline
		Special cases of  $S(n,k,\alpha ,\beta, \gamma)$  & \text{the special case}\\ \hline
	
		$S(n,k, 0,1,0)$  & \text{The Stirling numbers of the second kind}  $S(n,k)$    (see \cite{Comtet1974}) \\ \hline 
			$S(n,k,-\alpha, 0,0)$ & \text{Translated $r$-Whitney numbers} $W_{\alpha} (n,k)$ (see \cite{Belbachir2013}) \\ \hline
		
		$S(n,k, \alpha, 1, -r)$ & \begin{minipage}{8cm}
			The Howard degenerate weighted Stirling numbers $S(n,k,r|\alpha)$  (see \cite{Howard})\end{minipage}\\ \hline 
		$S(n,k, \alpha, 1, 0)$ & \text{The Carlitz degenerate Stirling numbers} $S(n,k|\alpha)$     (see \cite{Carlitz1979})\\ \hline 
		$S(n,k, 0, -1, r)$ & \text{The $r$-Stirling numbers} $S_r (n+r, k+r)$        (see \cite{Broader})\\ \hline 
		$S(n,k, 0, \beta, 1)$ & \text{The Whitney numbers} $W_{\beta} (n,k)$    (see \cite{MBouha}) \\ \hline 
		$S(n,k, 0, \beta, -r)$ & \text{The $r$-Whitney numbers} $W_{\beta, r} (n,k)$  (see \cite{Mezo2010})  \\ \hline 
		$S(n,k,0,0,1)$  & \begin{minipage}{8cm}
			\text{Binomial coefficients} $\binom{n}{k}$ (see \cite{benyi2022unfair})	\end{minipage} \\ \hline
		
			\end{tabular}
\end{table} 

\iffalse
\iffalse References to the literature should be numbered in square brackets like \cite{Burns-1995,debonothesis}.
The entries in the reference list should be in alphabetical order according to the first author listed. Also, please follow the way by which references are quoted at the end of this document. \fi 
\fi

\begin{definition}\cite{corcino2001}
Given $\alpha, \beta, r$ in non-negative integers such that $\alpha$ divides both $\gamma$ and $r$ . For $k + 1$ labelled cells such that the first $k$ cells each contain $\beta$ labelled compartments, and the $(k+1)^{\text{th}}$ cell contains $r$ labeled compartments. Where in each of the cells the compartments are having cyclic ordered numbering. Furthermore, the capacity of each compartment is limited to one ball. The number $\beta^k k! S(n, k; \alpha, \beta, r)$ gives the number of ways of distributing $n$ distinct balls one at a time into the $k + 1$  labelled cells such that only the cell having $r$ compartments may be empty.  
\end{definition}
On $S(n, k; \alpha, \beta, r)$  when $\alpha =0$, this is known as the non-degenerate case, and the numbers have the following combinatorial interpretation. 
\begin{definition} \cite{Adell2023}
	The numbers ${\beta}^k k! S(n,k,0,\beta,r)$ represent the number of partitions of $[n]$ into $k+1$ labelled cells such that:
	
	\begin{itemize}
		\item each of the first $k$ cells has $\beta$ compartments (these will be referred to as ordinary cells), 
		\item and the $(k+1)th$ cell has $r$ compartments (this will be referred to as the special cell),
		\item on each cell the compartments are cyclically ordered, 
		\item the capacity of each compartment is unlimited,
		\item the $(k+1)th$ cell is the only cell that may be empty.
	\end{itemize}

\end{definition}
% \begin{definition}
%${\beta}^k k!$ counts the number of ways to order $k$ ordinary cells and to choose one favorite compartment in each of the $k$ ordinary cells, given that each cell has $\beta$ compartments.
% \end{definition} 
Corcino in \cite{corcino2011} defined the following generalized Bell polynomials.
\begin{definition}For $\alpha,\beta$ and $\gamma$ real or complex
    \begin{equation}
        B_n (x,\alpha, \beta, r) = \sum_{k=0}^{n} S(n,k;\alpha, \beta, r)x^k  ,
    \end{equation} 
    and their generating function as, 
    \begin{definition}
        \begin{equation}
           \sum_{n=0}^{\infty} B_n (x,\alpha, \beta, r) \frac{z^n}{n!} = (1+\alpha z)^{\frac{r}{\alpha}} \exp\biggl[rz + \frac{(1+\alpha z)^{\frac{\beta}{\alpha}}-1}{\beta}\biggr].
        \end{equation}
    \end{definition}
\end{definition}
Corcino and Corcino in  \cite{corcino2011} explored many properties of the polynomials $B_n (x,\alpha, \beta, r)$ and obtained a number of theor identities using  algebraic methods. In this study we will provide  a combinatorial interpretation of the polynomials, prove combinatorially several of their identities, and provide their integral representations.

\begin{definition}\cite{gyimesi2019new}
	$r$-Whitney numbers of the second kind ${W}_{m,r}(n,k)$ may be interpreted combinatorially in the following way.  As the number of colored partitions of $[n+r]$ into $k+r$ subsets such that the following conditions are satisfied:
	\begin{itemize}
		\item the first $r$ elements of the set $[n+r]$ are in distinct blocks, 
		\item the blocks having the distinguished elements are not colored (the $r$ blocks will be referred to as distinguishable blocks), 
		\item 	 in each of the blocks that do not contain the distinguished elements the smallest element is the only element of the block that is not colored ,
		  all other elements within the block are colored with one of $m$ colors independently (such blocks will be referred to as non-distinguishable blocks).
		
	\end{itemize}
\end{definition}
\begin{definition}\cite{corcino2018some,gyimesi&nyul2018}
	$r$-Dowling polynomials ${D}_{m,r}^x(n)$ may be defined in the following ways 
	
		\begin{equation}\label{equation:220}
		\sum\limits_{n=0}^\infty{D}_{m,r}^x(n)\frac{z^n}{n!}=exp\begin{bmatrix}
			rz+\frac{x(e^{mz}-1)}{m}
		\end{bmatrix},
	\end{equation}
	\begin{equation}\label{equation:4}
		{D}_{m,r}^x(n)=\sum\limits_{k=0}^n{W}_{m,r}(n,k)x^k.
	\end{equation}

	%\begin{equation}
	%	\mathcal{D}_{m,r}(n)=\sum\limits_{k=0}^n\mathcal{W}_{m,r}(n,k).
	%\end{equation}
\end{definition}

Dowling numbers seems to first appear in Dowling's paper \cite{TDowling} in connection with geometric lattices of finite groups. A combinatorial interpretation of the polynomials ${D}_{m,r}(n)$ follows from that of the Whitney numbers ${W}_{m,r}(n,k)$ above in the following way.

\begin{definition}\cite{corcino2018some,gyimesi&nyul2018} The $r$-Dowling polynomials ${D}_{m,r}(n)$ give the number of colored partitions of $[n+r]$ into $k+r$ subsets (where $k$ runs from 0 to $n$) such that the following conditions are satisfied:
	\begin{itemize}
		\item the first $r$ elements of the set $[n+r]$ are in distinct blocks, 
		\item the blocks having the distinguished elements are not colored (the $r$ blocks will be referred to as distinguishable blocks), 
		\item 	 in each of the blocks that do not contain the distinguished elements the smallest element is the only element of the block that is not colored ,
		all other elements within the block are colored with one of $m$ colors independently (such blocks will be referred to as non-distinguishable blocks);
		\item each of the non-distinguishable blocks is independently colored with one of $x$ colors.
		
	\end{itemize}
\end{definition}

In this paper, by making use of the generalized stirling numbers $S(n,k,\alpha, \beta, \gamma)$ we define a new higher order degenerate generalization of the Dowling polynomials. We prove several of their combinatorial identities. Furthermore we represent closed forms of these higher order $r$-Dowling polynomials in terms of the bell polynomials $ B_n (x,\alpha, \beta, r)$, and provide combinatorial interpretations of these closed form expressions. We explore both their degenerate and their non-degenerate versions. The paper is structured in the following way; in section \ref{sec-1} we discuss a combinatorial interpretation of the bell polynomials $B_n (x,\alpha, \beta, r)$. In section \ref{sec-2} we discuss non-degenerate case of the higher order $r$-dowling polynomials. In  section \ref{sec-3} we explore  the degenerate version of the  higher order $r$-Dowling polynomials including their combinatorial interpretation and integral representations. The higher order $r$-Dowliing polynomials are the main object of this study. In section \ref{section*} we explore asymptotics results for the higher order $r$-dowling polynomials. 
\section{\text{Generalized Stirling numbers and Bell polynomials}}\label{sec-1}
In this section we interpret some of results derived in \cite{corcino2011} combinatorially. We categorize this section into two subsections focusing on non-degenerate and degenerate cases. The results are interpreted in terms of 'Unfair' distributions. Also, an integral representation of non-degenerate Bell Polynomials is provided.
\subsection{Non-degenerate case}
\begin{theorem}For $n,m,x \in \mathbb{Z^{+}}$ and $k\geq 0$, 
\begin{equation}
	B_{n+1}(x,m,r) = rB_{n}(x,m,r) + xB_{n}(x,m,m+r) .
\end{equation}
\end{theorem}
\begin{proof}
	This proof will be based on two cases, depending on the location of the ball. The two cases consider whether the ball is placed in a special cell or it is an $m$ labeled cell.\\
	
\textbf{Case 1:} If the $(n+1)^{\text{th}}$ ball is placed in the special cell with $r$ labeled compartments, the ball will have $r$ possibilities to occupy any of the $r$ labeled cells. The remaining $n$ balls will go to $(r,m,x)$ partitions in $B_{n}(x,m,r)$ ways.\\
	
\textbf{Case 2:} \iffalse We first form a cell, we call this cell $B$. This cell has $m$ compartments, and can be colored with one of $x$ colors in $x$ ways. The $(n+1)\text{-th}$ ball is placed in $B$. Thereafter, $B$ is merged with special cell to form a Unit with $m+r$ compartments, if the ball occupies Block $B$, the remaining $n$ balls can go to other compartments within the unit including the compartment with $(n+1)\text{-th}$ ball since there are no restrictions on how many balls each compartment can contain. This is done in $B_{n}(x,m,m+r)$ ways. \fi 
In this case, we first consider forming a unit of $m+r$ cells which will be a result of combining the $m$ labeled cells with special $r$ labeled cell. We then place the $(n+1)^{\text{th}}$ ball in the unit, which also qualifies to be colored with any of $x$ colors. Since there are no restrictions on the unit, the $n$ balls may go occupy any compartment including the already preoccupied compartment in $B_{n} (x,m,m+r)$ ways.  
\end{proof} 

\begin{theorem}For $n,m,x \in \mathbb{Z^{+}}$ and $k\geq 0$,
	\begin{equation}
		B_{n}(x,m,r)= \sum_{k=0}^{n} \binom{n}{k} B_k (x,m,0) r^{n-k} .
	\end{equation}
\end{theorem} 
\begin{proof}
In this proof, we want to count the number of ways at which $n$ balls can be distributed to both $m$ and $r$ labeled cells. Before we start counting, we note that $r$ labeled cell is somehow, a special cell, which cannot be colored by any color. This special cell is $r$ labeled because it is further divided into $r$ compartments, similarly  with $m$ labeled cells. Now, we start to distribute some $k$ balls  into $m$ labeled cells, these balls are selected from $n$ balls  in ${n\choose k}$ ways. The distribution of $k$ balls to $(r,m,x)$ partitions is counted by $ B_k (x,m,0)$. Next, the remaining $n-k$ balls will be distributed to the $r$ labeled cell, in $r^{n-k}$ ways. Then we sum up all the possibilities. 
\end{proof}
The identities in Theorem \ref{Corc.2} were proposed by Corcino and Corcino in \cite{corcino2011} algebraically, we give a combinatorial interpretation for Corcino's results.
\begin{theorem} For $n,m,x \in \mathbb{Z^{+}}$ and $k\geq 0,$ \label{Corc.2}
    \begin{equation} 
		B_{n+1} (x,m,r) =rB_{n} (x,m,r)+x\sum_{k=0}^{n}\binom{n}{k}{m}^{k} B_{n-k} (x,m,r) .
	\end{equation}
\end{theorem}
\begin{proof}
    This proof relies on the position of $(n+1)^{\text{th}}$ ball, leading us to consider two cases.\\ 

    \textbf{Case 1:} Suppose the $(n+1)^{\text{th}}$ ball is placed into a labeled cell. In this case, there are no issues with assigning the first ball, this enables any of the $r$ compartments to be selected for this single ball. This can be done in $r$ ways. The remaining $n$ balls are then distributed across the remaining cells, including the already occupied cell, in $B_n (x,m,r)$ ways \\

    \textbf{Case 2:} Next, we consider the $(n+1)^{\text{th}}$ ball being distributed among the $m$ labeled cells, say cell $D$. If this ball lands in the favorite compartment, the occupied cell may be colored in $x$ ways. We then select $k$ balls from the $n$ balls in $n\choose k$ ways. These $k$ balls will be assigned to the $m$ compartments of cell $D$ independently. The term $B_{n-k}(x.m,r)$ give the number of ways to construct $m$ and $r$ labeled cells. The remaining $n-k$ balls will be distributed among $r$ labeled cells and $m$ labeled cells in $B_{n-k}(x,m,r)$ ways.
\end{proof}
\begin{theorem}For $n,m,x \in \mathbb{Z^{+}}$ and $k\geq 0,$
 	\begin{equation}
 		B_n (x,m,r) =\frac{2n!}{\pi e^{x/m}} Im\int_{0}^{\pi} e^{e^\frac{xmi\theta}{m}} e^{re^{i\theta}} sin(n\theta) d\theta
 	\end{equation}
 \end{theorem}
 \begin{proof}
 	We make use of Cassado's identity in \cite{Casssado}: \[Im \int_{0}^{\pi} e^{je^{i\theta}}sin(n\theta)d\theta =\frac{\pi}{2} \frac{j^n}{n!}\] 
 	\[j^n =\frac{2n!}{\pi}Im \int_{0}^{\pi} e^{je^{i\theta}}sin(n\theta)d\theta\]
 	\[S(n,k,x,m,r)=x^k \frac{1}{m^k k!}\sum_{j=0}^{k}\binom{k}{j}(-1)^{k-j}(mj+r)^n\]
 	\[S(n,k,x,m,r)=\frac{x^k}{m^k k!}\sum_{j=0}^{k}\binom{k}{j}(-1)^{k-j}\frac{2n!}{\pi}Im \int_{0}^{\pi} e^{(mj+r)e^{re^{i\theta}}}sin(n\theta)d\theta\]
 	\[S(n,k,x,m,r)=\frac{2n!}{\pi}\frac{x^k}{m^k k!}Im \int_{0}^{\pi} (e^{me^{i\theta}}-1)^k  e^{re^{i\theta}}sin(n\theta)d\theta\] therefore, \[\sum_{k=0}^{\infty}S(n,k,x,m,r)=\frac{2n!}{\pi}Im \int_{0}^{\pi} \sum_{k=0}^{\infty}\biggl(\frac{x(e^{me^{i\theta}}-1)}{m}\biggr)^k \cdot\frac{1}{k!}(e^{re^{i\theta}}sin(n\theta)d\theta)\] 
 	\[=\frac{2n!}{\pi e^{x/m}} Im\int_{0}^{\pi} e^{e^\frac{xmi\theta}{m}} e^{re^{i\theta}} sin(n\theta) d\theta\]
 \end{proof}
 \section*{\text{Unfair Distributions in general case}}
 In this section, we look at the combinatorial structure for unfair distributions as outlined by Beáta , Nkonkobe and Shattuck in \cite{benyi2022unfair}. It is noted that the allocation of distinguishable balls into blocks, which are further subdivided into compartments, is governed by a parameter $\alpha$. When $\alpha >0$, each ball placed in a compartment eliminates that compartment and the subsequent $\alpha - 1$ compartments, ensuring no additional balls can be placed there. These compartments within the blocks are then referred to as cyclically ordered. Conversely, if $\alpha = 0$, the compartments are allowed to hold more than one ball. These two scenarios correspond to the following combinatorial models:
\begin{itemize}
    \item $x^n$, which give the number of ways to distribute $n$ distinguishable balls into $x$ distinguishable blocks.
    \item $(x)_n$, enumerates the number of ways of assigning $n$ distinguishable balls to $x$ distinguishable cells, with each cell containing at most one ball.
\end{itemize}    
These cases are understood to align with $\alpha = 0$ and $\alpha =1$ , respectively (see \cite{benyi2022unfair}). The concept of “unfair” comes from the premise that within a block, there exists a favored compartment where the first ball is preferentially placed, leading to what is termed an 'unfair distribution' \cite{benyi2022unfair}.
\begin{definition}(\cite{benyi2022unfair})[Unfair Distributions] \label{definition:2021} 
 	The numbers $S(n,k;\alpha,m,r)$ give the number of distributions of $[n]$ into $k+1$ distinct cells satisfying the following conditions:
 	\begin{itemize}
 		\item  each of the $k$ cells has $m$ compartments, 
 		\item the other cell has $r$ compartments,
 		\item each of the $k+1$ cells has cyclic ordered numbering
 		\item in each cell the first element goes into the favorite compartment. 
 		\item in each placement of elements into compartments, the first element in each cell always goes into the favorite compartment, 
 		\item each of the $k$ cells are colored independently with one of $x$ colors.
 	\end{itemize} 
 \end{definition}
The number of Unfair distributions satisfy the following identities,
 \begin{theorem}\label{UnfairD}(\cite{benyi2022unfair}). For $n,\alpha ,r \geq 0,$ and $m\geq 1$, 
  \begin{equation}
  	S(n+1,k; \alpha, m, r) =S(n,k-1; \alpha, m, r) + (km-n\alpha +r)S(n,k; \alpha, m, r) . 
  \end{equation}
 \end{theorem}
 \begin{theorem}\label{UnfairD2}(\cite{benyi2022unfair}). For $n,\alpha ,r \geq 0,$ and $m\geq 1$
 	\begin{equation}
 		kS(n,k; \alpha, m, r) = \sum_{j=k-1}^{n-1} \binom{n}{j}(m-\alpha|\alpha)_{n-j-1}S(j,k-1; \alpha, m, r) .
 	\end{equation}
 \end{theorem}
 
 \subsection{Degenerate case}
 \par 
 In \cite{corcino2011}, the following theorem is defined and the identities were derived algebraically. In this paper, we provide a combinatorial proof for their results.

 \begin{theorem} For $n,\alpha ,r \geq 0,$ and $m,x\geq 1$ we have
	\begin{equation}
		S_{n+1}(x,\alpha,m,r) =(r-\alpha n)S_{n}(x,\alpha, m,r) + x\sum_{k=0}^{n} \binom{n}{k} (m|\alpha)_k S_{n-k}(x,\alpha,m,r).
	\end{equation}
\end{theorem}
\begin{proof}
    This proof relies on the placement of $(n+1)^{\text{th}}$ ball, thus we consider two cases;\\

    \textbf{Case 1:} We consider the case of $(n+1)^{\text{th}}$ ball occupying an $r$ labeled cell. The ball will first be placed at a preferred compartment, after its placement, the next $\alpha$ compartments will close, making $r-\alpha n$ compartments available for occupation. The remaining $n$ elements may occupy any available compartments in the special cell and in the $m$ labeled cell, this is be counted by $S_{n}(x,\alpha, m,r)$.\\ 

    \textbf{Case 2:} Now, we consider the case of the $(n+1)^{\text{th}}$ ball in a non-distinguished cell. All non distinguished cells  are subject to coloring, and may be colored with any of $x$ colors. Next, after preferentially placing the  $(n+1)^{\text{th}}$ ball in the cell, next $\alpha$ compartments closed. From the remaining $n$ balls we chose $k$ balls in $\binom{n}{k}$ ways then we distributed $k$ balls specifically to the already occupied $m$ labeled cell in $(m|\alpha)_k$ ways. All the remaining $n-k$ balls fill the remaining compartments in the $m$ labeled cell and $r$ labeled cell, this is counted by $S_{n-k}(x,\alpha,m,r)$.
\end{proof}
\begin{theorem}For $n,\alpha ,r \geq 0,$ and $m,x\geq 1$ we have
	\begin{equation}
		S_{n+1}(x,\alpha,m,r) = rS_{n}(x,\alpha,m,r-\alpha) +xS_{n}(x,\alpha,m,m+r-\alpha).
	\end{equation}
\end{theorem}
\begin{proof}
    To give the recurrence relation  $S_{n+1}(x, \alpha, m, r)$ , we start our proof by examining two distinct cases based on the placement of the $(n+1)^{\text{th}}$ ball.\\

\textbf{Case 1}: Suppose the $(n+1)^{\text{th}}$ ball is placed in the special cell, which is further divided into  $r$  compartments. This ball can occupy any one of the  $r$  compartments, providing  $r$  possible placement options. Once placed, the $(n+1)^{\text{th}}$ ball occupies one compartment and closes the next  $\alpha$ available compartments within the special cell, reducing the number of available compartments to  $r - \alpha$. The remaining  $n$ balls are then distributed among the $m$ labeled cells and the special cell (now with  $r - \alpha$  available compartments), which can be counted in  $S_n (x, \alpha, m, r - \alpha)$  ways. Thus the total number of ways for this case is  $r \cdot S_n (x, \alpha, m, r - \alpha)$.\\

\textbf{Case 2}: Now we consider the case where the $(n+1)^{\text{th}}$ ball is placed into one of the first  $k$ labeled cells. We choose one cell, and denote it as $B$. This cell can be colored with one of any $x$ distinct colors in $x$ ways. Next, we combine $B$ with the special cell, which contains $r$ compartments, into a unit. The placement of the $(n+1)^{\text{th}}$ ball in cell  $B$  reduced the number of available compartments by $\alpha$, leaving $m + r - \alpha$ compartments available for the distribution of the remaining $n$ balls in the entire unit. The distribution of  $n$  balls into the newly formed unit is counted by $S_n (x, \alpha, m, m + r - \alpha)$  ways. Thus total number of ways to distribute $n$ balls for this case where the $(n+1)^{\text{th}}$ ball is placed in a unit is counted by $x\cdot S_n(x, \alpha, m, m + r - \alpha)$.
\end{proof}
\begin{theorem}For $n,\alpha ,r \geq 0,$ and $m,x\geq 1$ we have
	\begin{equation}
		S_{n+1}(x,\alpha,m,r) = rS_{n}(x,\alpha,m,r-\alpha) +x\sum_{j=0}^{n}\binom{n}{j}(m+r|\alpha)_j B_{n-j}(x,m,\alpha).
	\end{equation}
\end{theorem}
\begin{proof}
To establish the recurrence relation $S_{n+1}(x,\alpha,m,r)$, we consider two scenarios based on the placement of the $(n+1)^{\text{th}}$ ball. \\ 

\textbf{Case 1:} We assume that the $(n+1)^{\text{th}}$ ball is preferentially placed on the special cell containing $r$ compartments with the ball having $r$ possible placement options, after the ball is placed, the next $\alpha$ compartments will close thus resulting in $r-\alpha$ compartments being the only available compartments in the special cell for preoccupation. The distribution of the remaining $n$ balls will fill the remaining $r-\alpha$ compartments then the $m$ labeled cell, this will be counted by $S_{n}(x,\alpha,m,r-\alpha)$.\\

\textbf{Case 2:} In this case, we first preferentially place the $(n+1)^{\text{th}}$ ball in an $m$ labeled cell, denoted by $B$ which is colored by any of the $x$ colors. After the placement of the ball in $B$, the next $\alpha$ compartments close resulting in $r-\alpha$ available compartments in $B$. Thereafter,  we select $j$ balls out of the remaining $n$ balls in $\binom{n}{j}$ ways, we then combine $B$ with the special cell consisting of $r$ compartments into a unit of $m+r-\alpha$ compartments and thereafter, we place the $j$ balls into the unit where the distribution of these $j$ balls is given by $(m+r|\alpha)_j$. All remaining $n-j$ balls will be distributed in the remaining $m$ labeled cells in $B_{n-j} (x,m,\alpha)$ ways, and then sum all the possibilities. Hence the total number of distributions in this case is $x\sum_{j=0}^{n}\binom{n}{j}(m+r|\alpha)_j B_{n-j}(x,m,\alpha).$
\end{proof}
\begin{theorem}For $n,\alpha ,r \geq 0,$ and $m,x\geq 1$ we have
	\begin{equation}
		S_{n}(x,\alpha,m,r) =\sum_{j=0}^{n}\binom{n}{j}(r|\alpha)_{j} B_{n-j}(x,\alpha,m) 
	\end{equation}
\end{theorem}
\begin{proof}
  The Stirling number $S_n (x, m, r)$ give the number of ways to distribute $n$ distinct balls into  $k + 1$  cells. To begin the process of counting these distributions, we select $j$  balls from the  $n$  available balls, which can be done in  $\binom{n}{j}$  ways. Next, we distribute these  $j$  selected balls into the special cell that contains  $r$  compartments. The placement of the balls in the special cell follows 'Unfair distributions' rule: the first ball is placed preferentially, after which it occupies a compartment and subsequently closes the next  $\alpha$ available compartments. This process continues until all $j$ balls fill the available compartments, with the number of ways to distribute the  $j$  balls into this special cell, given its $r$ compartments, being given by $(r|\alpha)_j$. The remaining $n-j$ balls are distributed to the $m$ labeled cells in $B_{n-j}(x,\alpha,m) $ ways. 
\end{proof}

\section{$r$-Dowling polynomials}\label{sec-2}
In this section we introduce higher order non-degenerate $r$-Dowling polynomials. We interpret our results combinatorially based on the notion of 'unfair' distributions and a framework of $(r,m,x,\lambda)$ partitions. Additionally, we provide an integral representation of the higher order non-degenerate $r$-Dolwing polynomials
\subsection{Non-Degenerate case}

\begin{definition}\label{Whitney number Part.}
	Given a set $A_{n,r}=\{1,\ldots, r, \ldots,n+r\}$. Let $n,r\geq 0$ and $\lambda,m,x\in \mathbb{Z^{+}}$ and all the partitions of the set $A_{n,r}$ are written in standard form. Then $D^{\lambda,x}_{m,r}(n)$ is the number of partitions of $A_{n,r}$ such that: 
	\begin{itemize}
		\item The $r$ distinct elements $1,2,\ldots,r$ form $r$ distinct blocks 
		\item All the elements but the last one in non-distinguished blocks are colored with one of $m$ colors independently (note that the last element is also maximal in the block, as per the standard form) 
		\item The non-distinguished blocks are colored with one of the $x$ colors independently
		\item Neither the elements in the distinguished blocks nor the distinguished blocks are colored
	\end{itemize}
	The partitions of $A_{n,r}$ satisfying the above assumptions will be called $(r,m,x,\lambda)$-partitions.
\end{definition}

\begin{theorem}\label{thm-1}For $n,\lambda ,r \geq 0,$ and $m,x\geq 1$ we have
\begin{equation}\label{equation:6}
	D^{\lambda,x}_{m,r}(n+1)=rD^{\lambda,x}_{m,r}(n) +x\lambda\sum_{i=0}^{n} \binom{n}{i} m^i D^{\lambda,x}_{m,r}(n-i).
\end{equation}
\end{theorem}

\begin{proof}
\iffalse Suppose that $n,m,x \geq 1$	and $r\geq 0$\fi To construct an $(r,m,x,\lambda)$ partition, we consider two cases based on the placement of $(n+1)^{\text{th}}$ element. \\
 
 \textbf{Case 1:} Suppose the $(n+1)^{\text{th}}$ element forms a distinguished block, there are $r$ possible ways it can form these blocks. The remaining $n$ elements may be distributed to the $(r,m,x,\lambda)$ partition in $	D^{\lambda,x}_{m,r}(n)$ ways. \\
 
 \textbf{Case 2:} If the $(n+1)^{\text{th}}$ element is a non-distinguished element, it will occupy a non-distinguished block, say block $G$. The block may be colored with any of the $x$ colors in $x$ ways, and may be sectioned in $\lambda$ sections in $\lambda$ ways. To fill block $G$, we select some $i$ elements from $n$ elements in $\binom{n}{i}$ ways. The $i$ elements will then be distributed to block $G$ containing $(n+1)^{\text{th}}$ where they are arranged in increasing order. All the non-maximal $i$ elements then will be colored with one of the $m$ colors independently in $m^i$ ways. The remaining $n-i$ elements will then be distributed to $(r,m,x,\lambda)$ partition in  $D^{\lambda,x}_{m,r}(n-i)$ ways. This completes the proof.
	
\end{proof}
\begin{theorem}\label{Thrm2}
	 For $n,\lambda ,r \geq 0,$ and $m,x\geq 1$ we have
	\begin{equation}\label{equation:10}
		D^{\lambda,x}_{m,r}(n+1)= rD^{\lambda,x}_{m,r}(n) +x\lambda D^{\lambda,x}_{m,m+r}(n).
	\end{equation}
\end{theorem}
\begin{proof}
We interpret the result in terms of unfair distributions. We use the same approach as in Theorem \ref{thm-1} where two cases based on the distribution of the $(n+1)^{\text{th}}$ ball are considered.\\

\textbf{Case 1:} Suppose the $(n+1)^{\text{th}}$ ball is distributed to an $r$ labeled cell. The ball can be randomly placed in any of the $r$ compartments in $r$ ways. The remaining $n$ elements will be distributed among the other cells in $D^{\lambda,x}_{m,r}(n)$ ways.\\

\textbf{Case 2:} As a starting point, we create a new $m$ labeled cell, called $G$. The cell may be colored  with any of the $x$ colors in $x$ ways and sectioned among $\lambda$ sections in $\lambda$ ways. We then drop the $(n+1)^{\text{th}}$ ball on the cell, with the ball going to the favored compartment. Furthermore, we combine $G$ with the special cell to form a unit cell with $m+r$ compartments. All the remaining $n$ elements are distributed to the unit cell and all other cells in $D^{\lambda,x}_{m,m+r}(n)$ ways. This completes the proof.
\end{proof} 

	\begin{theorem}\label{Thrm3} For $n,\lambda ,r \geq 0,$ and $m,x\geq 1$ we have
	\begin{equation}\label{equation:10.1}
		D^{\lambda,x}_{m,r}(n+1) = rD^{\lambda,x}_{m,r}(n) +\lambda x\sum_{i=0}^{n}\binom{n}{i}(m+r)^i D^{\lambda, x}_{m,0}(n-i).
	\end{equation}
\end{theorem}
\begin{proof}
	Similar to the proof of Theorem \ref{thm-1} and Theorem \ref{Thrm2}, we base our proof on the placement of $(n+1)^{\text{th}}$ element. \\
	
	\textbf{Case 1:} If the $(n+1)^{\text{th}}$ element is placed in a distinguishable block, there are $r$ choices for the element to occupy any of the $r$ blocks. The remaining $n$ elements will be distributed among $(r,m,x,\lambda)$ partitions in $D^{\lambda,x}_{m,r}(n)$ ways. \\
	
	\textbf{Case 2:} We first create a non-distinguished block, and we call it block $B$, then we place the $(n+1)^{\text{th}}$ element in the block. The block can be colored with any of the $x$ colors in $x$ ways and sectioned among $\lambda$ sections in $\lambda$ ways. Next, we merge block $B$ with $r$ distinguished blocks to form a unit. Since $B$ is non-distinguished, all its non-maximal elements will be colored with any of $m$ colors in $m$ ways. Next, we merge block $B$ and $r$ blocks into a Unit of size $(m+r)$. From the remaining $n$ elements, we select $i$ elements in $\binom{n}{i}$ ways with $i$ elements distributed into the unit in $(m+r)^{i}$ ways. Remaining $n-i$ elements will form  an $(r,m,x,\lambda)$ partition in $D^{\lambda, x}_{m,0}(n-i)$ ways.
\end{proof}

\begin{theorem}  For $n,\lambda ,r \geq 0,$ and $m,x\geq 1$ we have
	\begin{equation}\label{equation:13}
		D^{\lambda,x}_{m,r}(n)=\sum_{e_{1} + e_2 +...+e_{\lambda +1}=n} \binom{n}{e_{1} , e_2 ,\ldots e_{{\lambda}+1}} (r)^{e_1}B_{e_2}(0,m,x)B_{e_3}(0,m,x)...B_{e_{\lambda +1}}(0,m,x).
	\end{equation}
\end{theorem} 
\begin{proof}
	In this proof, we distribute $n$ elements into the ${\lambda}+1$ sections. This can be done by first, dividing $n$ elements into  ${\lambda}+1$ sections where $e_j$ for $1\leq j \leq  {\lambda}+1$ being the sizes of each section. The number of ways to divide the $n$ elements into ${\lambda}+1$ sections is $\binom{n}{e_{1} , e_2 ,\ldots e_{{\lambda}+1}}$. Next, all elements in section $e_1$ are placed in the $r$ labeled cell in $r^{e_1}$ ways. The remaining elements from other sections are distributed into other cells, independently. These cells are colored with one of $x$ colors. Therefore, all elements in section $e_2$ are distributed into $m$ labeled cell, with the first element from the section going to the favored compartment, this can be done in $B_{e_2}(0,m,x)$ ways. Similar process is repeated for all sections until section ${\lambda}+1$. This completes the proof.
\end{proof}

\begin{theorem}\label{Dowling Th4} For $n,\lambda ,r \geq 0,$ and $m,x\geq 1$ we have
	\begin{equation}
		D^{\lambda,x}_{m,r}(n)=\sum_{i=0}^{n}\binom{n}{i}B_i (0, m,0,x){\lambda}^i r^{n-i}.
	\end{equation}
\end{theorem} 

\begin{proof}
	In this proof, we distribute $n$ balls into $k+1$ cells. All the $k$ cells are $m$ labeled. Each of the $m$ compartments can hold more than one ball. The $k$ cells are colored with one of $x$ colors. The special cell has $r$ compartments and it cannot be colored with any color. The distribution of $n$ balls into $k+1$ cells will happen this way: we first, separate the first $k$ cells and the $(k+1)^{\text{th}}$ cell, then we select some $i$ balls from a total of $n$ balls, this process can be done in $\binom{n}{i}$ ways. Next, we distribute each of the $i$ balls into $k$ cells such that no cell is allowed to be empty and that, the $k$ cells will be colored by any of $x$ colors, independently. This can be done in $\sum_{i}^{k} S(i,k,0, m,0) x^k$  ways, hence the total distribution of these $i$ balls into $k$ cells is counted by $B_i (0, m,0,x)$. All the $i$ balls are sectioned in $\lambda$ sections independently. The $n-i$ remaining balls will line up the special cell in $r^{n-i}$ ways. Then we complete the proof.
	\end{proof}
	
	\begin{theorem}\label{Dowling Th3} For $n,\lambda ,r \geq 0,$ and $m,x\geq 1$ we have
		\begin{equation}\label{equation:7}
			D^{\lambda,x}_{m,r}(n)= \sum\limits_{i=0}^{n}(\lambda x)^i  S(n,i;0,m,r).
		\end{equation}
	\end{theorem}
	\begin{proof}
		In the proof, we want to distribute $n$ balls into $i+1$ cells, with the first $i$ cells having $m$ compartments and the $(i+1)^{\text{th}}$ cell is a special cell with $r$ compartments. In the distribution, more than one ball is allowed to occupy each compartment. The distribution will be counted by $S(n,i;0,m,r)$. All the $m$ labeled cells will be colored with any of $x$ colors independently in $x^{i}$ ways. Also, the cells will be sectioned among $\lambda$ sections in ${\lambda}^i$ ways. We then sum up all the possibilities. This completes the proof.
	\end{proof} 
	\begin{theorem} \label{Generating function for Dowling n} For $n,\lambda ,r \geq 0,$ and $m,x\geq 1$ we have
		\[d^{\lambda,x}_{m,r}(n)=\sum_{n=0}^{\infty}D^{\lambda,x}_{m,r}(n)\frac{z^n}{n!} =exp\bigg(rz+\frac{\lambda x}{m}(e^{mz}-1)\bigg).\] 
	\end{theorem}
	\begin{proof} We know from Theorem \ref{Dowling Th4} and Equation \ref{equation:220} that $	D^{\lambda,x}_{m,r}(n)=\sum_{s=0}^{n}\binom{n}{s}B_n(x,0, m, 0){\lambda}^s r^{n-s}$. This implies that \[d^{\lambda,x}_{m,r}(n)=\sum_{n=0}^{\infty}\frac{z^n}{n!}\biggl\{\sum_{i=0}^{n}\binom{n}{i}B_i (0, m, 0, x){\lambda}^i r^{n-i}\biggr\},\] \[d^{\lambda,x}_{m,r}(n)=\sum_{n=0}^{\infty}\biggl[\sum_{k+i=n}^{}r^k z^{k+i}{\lambda}^i B_i (0, m, 0,x)\frac{1}{i!k!}\biggr],\] \[d^{\lambda,x}_{m,r}(n)=\sum_{k=0}^{\infty}\frac{(rz)^k}{k!}\sum_{i=0}^{\infty}{\lambda}^i B_i (0,m,0,x)\frac{z^i}{i!},\] \[d^{\lambda,x}_{m,r}(n)=e^{rz}\sum_{i=0}^{\infty} ({\lambda}x)^i \biggl(\frac{e^{mz}-1}{m}\biggr)^i\frac{1}{i!},\]  
		now \[D^{\lambda,x}_{m,r}(n)= e^{rz}\exp{\biggl[x\lambda\biggl(\frac{e^{mz}-1}{m}\biggr)}\biggr]\] 
		we know by Equation \ref{equation:220}, for $r=0$. 
	\end{proof}
	The following theorems are a result by Gymesi and Nyul in \cite{gyimesi&nyul2018}. We generalize the results for arbitrarily $\lambda$ in nonnegative integers and provide their combinatorial results. 
	\begin{theorem}
		For $n,\lambda ,r \geq 0,$ and $m,x\geq 1$ we have 
		\begin{equation}
			D^{\lambda, x}_{m,r}(n)=\sum_{j=0}^{n}\binom{n}{j}{\lambda}^{n-j}(r-s)^{n-j}D^{\lambda,x}_{m,s}(j).
		\end{equation}
		\end{theorem} 
		\begin{proof}
			In a set of $[n+r]$ elements, the first $r$ elements are distinct such that they form $r$ distinguishable blocks. From $r$ blocks, $s$ blocks are taken to form $(s,m,x,\lambda)$ partition and $(r-s)$ blocks remain. Next, $j$ elements are selected from $n$ elements in $\binom{n}{j}$ ways. Selected elements will then go fill the $(s,m,x,\lambda)$ in $D^{\lambda,x}_{m,s}(j)$ ways. Remaining $(n-j)$ elements will be distributed to the $(r-s)$ blocks, independently in $(r-s)^{n-j}$ ways. Also, these elements are sectioned across $\lambda$ sections in ${\lambda}^{n-j}$ ways. We then sum up all the possibilities and complete the proof.
		\end{proof}
		
		\begin{theorem}
			For $n,\lambda ,r \geq 0,$ and $m,x\geq 1$ we have
			\begin{equation}
				D^{\lambda,x}_{m,r}(n+1)=r\lambda D^{\lambda,x}_{m,r}(n) +x\lambda \sum_{j=0}^{n} \binom{n}{j}D^{\lambda,x}_{m,r} (j)m^{n-j}.
			\end{equation}
		\end{theorem}
		\begin{proof}
In this proof, we interpret the enumeration of $D^{\lambda,x}_{m,r}(n+1)$ based on the placement of the $(n+1)^{\text{th}}$ element.\\

\textbf{Case 1:} If  $(n+1)^{\text{th}}$ element is in a distinguished block, it can occupy any of the $r$ blocks in $r$ ways. Similarly, the element could occupy any of the $\lambda$ sections in $\lambda$ ways. The remaining $n$ elements go and occupy the $(r,m,x,\lambda)$ partitions in $D^{\lambda,x}_{m,r}(n)$ ways. \\ 

\textbf{Case 2:} A non-distinguishable block which we name block $A$ is colored with one of the $x$ colors in $x$ ways and sectioned among $\lambda$ sections in $\lambda$ ways. If $(n+1)^{\text{th}}$ element is in block $A$, it immediately becomes the maximal element of the block. All elements filling block $A$ will be non maximal and be subjected to coloring. Then we will have $n$ elements remaining. From the $n$ elements, we select $j$ elements in $\binom{n}{j}$ ways. The $j$ elements will occupy $(r,m,x,\lambda)$ partitions in $D^{\lambda,x}_{m,r} (j)$ ways. The remaining $n-j$ elements will then go fill block $A$, these elements will be colored with $m$ colors in $m^{n-j}$ ways. This completes the proof.
		\end{proof}

\begin{theorem}	For $n,\lambda ,r \geq 0,$ and $m,x\geq 1$ we have
	\begin{equation}
		D^{\lambda,x}_{m,r}(n)= \sum_{i=0}^{n}\biggl[\sum_{s=i}^{n}\binom{n}{s}r^{n-i}m^{s-i}S(s,i)\biggr]x^i .
	\end{equation}
\end{theorem}
\begin{proof}
	In this proof we have an $[n+r]$ element set where the first $r$ elements are distinct, these elements will then form a singleton subsets thus formulating $r$ blocks, leaving a set of $n$ elements. Firstly, some $s$ elements are selected from $n$ elements in $\binom{n}{s}$ ways. These $s$ elements are then distributed into $i$ blocks in $S(s,i)$ ways. Each of the $i$ blocks has a maximal element. Next, we color all non-maximal elements with one of $m$ colors in $m^{s-i}$ ways. The remaining $n-i$ elements will be placed in the distinguishable in $r^{n-i}$ ways. All the non-distinguishable blocks can be colored with any of the $x$ colors in $x^i$ ways. We sum over all the possibilities. Then we complete the proof.
\end{proof}		

\begin{theorem}For $n,\lambda ,r \geq 0,$ and $m,x\geq 1$ we have
	\begin{equation}
		D^{\lambda, x}_{m,r}(n)= \sum_{k=0}^{n}\biggl[\sum_{e_0 +e_1 +\cdots + e_k =n-k}^{}r^{e_0}(r+m)^{e_1}\cdots (r+km)^{e_k}\biggr]x^k .
	\end{equation}
\end{theorem}

	\begin{proof}
	In this proof, we distribute $n-k$ elements among $k$ blocks of size $e_j$, for $0\leq j\leq k$. The $k$ blocks may be colored with any of the $x$ colors in $x^k$ ways. Firstly, all $n-k$ elements in the block are distributed among the blocks of sizes $e_j$. The first distribution of elements in block $e_0$ happens in special cell $r$, in $r^{e_0}$ ways. Remaining elements from other blocks are then distributed to the unit $(r+m)$, for instance, elements in block $e_1$ of size 1 are are to be distributed in the $(r+m)$ block in $(r+m)^{e_1}$ ways. The process can be repeated until the last block $e_k$ of size $k$. Then we complete the proof.
\end{proof}

\begin{theorem}For $n,\lambda ,r \geq 0,$ and $m,x\geq 1$,
	\begin{equation}
		D^{x,\lambda}_{m,r}(n) =\frac{2n!}{\pi e^{x\lambda /m}} Im \int_{0}^{\pi} e^{\frac{e^{x\lambda mi\theta}}{m}} e^{re^{i\theta}}sin(n\theta)d\theta
	\end{equation}
\end{theorem} 	
\begin{proof} We use the following identity \cite{Casssado}: 
	\[Im \int_{0}^{\pi} e^{je^{i\theta}}sin(n\theta)d\theta =\frac{\pi}{2} \frac{j^n}{n!}\] 
	\[j^n =\frac{2n!}{\pi}Im \int_{0}^{\pi} e^{je^{i\theta}}sin(n\theta)d\theta\]
	\[W_{m,r}^{\lambda, x}(n,k)=(\lambda x)^k \frac{1}{m^k k!}\sum_{j=0}^{k}\binom{k}{j}(-1)^{k-j}(mj+r)^n\] \[W_{m,r}^{\lambda, x}(n,k)=(\lambda x)^k \frac{1}{m^k k!}\sum_{j=0}^{k}\binom{k}{j}(-1)^{k-j}\frac{2n!}{\pi} Im\int_{0}^{\pi} e^{je^{i\theta}}sin(n\theta)d\theta\]
	\[W_{m,r}^{\lambda, x}(n,k)=(\lambda x)^k \frac{1}{m^k k!}\sum_{j=0}^{k}\binom{k}{j}(-1)^{k-j}\frac{2n!}{\pi} Im\int_{0}^{\pi} e^{(mj+r)e^{i\theta}}sin(n\theta)d\theta\]
	\[=(\lambda x)^k \frac{2n!}{\pi}\frac{1}{m^k k!}\sum_{j=0}^{k}\binom{k}{j}(-1)^{k-j} Im\int_{0}^{\pi} (e^{me^{i\theta}}-1)^k e^{re^{i\theta}}sin(n\theta)d\theta\] Therefore, 
	\[\sum_{k=0}^{\infty}W_{m,r}^{\lambda, x}(n,k)=\frac{2n!}{\pi}Im\int_{0}^{\pi}\biggl(\sum_{k=0}^{\infty}\frac{(e^{me^{i\theta}}-1)^k (\lambda x)^k}{m^k k!}\biggr)e^{re^{i\theta}}sin(n\theta)d\theta\] 
	\[=\frac{2n!}{\pi e^{x\lambda/m}}Im\int_{0}^{\pi}e^{\frac{e^{x\lambda mi\theta}}{m}} e^{re^{i\theta}}sin(n\theta)d\theta\]
	
	The method has been used in \cite{Casssado, corcino2011} 
	\end{proof}
\section{Degenerate Case}\label{sec-3} 
In this section we focus on interpreting the higher order degenerate $r$-Dowling polynomials by combinatrial means using both Unfair distribution and $(r,m,x,\lambda)$ partitions.\\
The generating function of the degenerate $r$-Dowling polynomials also simplifies:

\begin{equation}\label{Degenerate r-Dowling generating function equation}
	\sum_{n=0}^{\infty}D^{\lambda, x}_{m,r}(n;\alpha)\frac{z^n}{n!} = 	e_\alpha^{r}(z)e^{x\lambda[\frac{e^m_\alpha(z)-1}{m}]}
\end{equation}

\begin{theorem}For $n\geq 0$ and $\alpha, x, m,r  \in \mathbb{Z^{+}}$ such that $\alpha|m$ and $\alpha|r$ we have,
	\begin{equation} 
		d^{\lambda,x}_{m,r}(n;\alpha) = \sum_{n=0}^{\infty}D^{\lambda,x}_{m,r}(n;\alpha)\frac{z^n}{n!} = 	e_{\alpha}^{r}(z)e^{x\lambda[\frac{e^m_\alpha(z)-1}{m}]} .
	\end{equation}
	\begin{proof}
		we know know from Equation \ref{Degenerate r-Dowling generating function equation} that $D^{\lambda,x}_{m,r}(n;\alpha) = \sum_{k=0}^{n} (x\lambda)^k S(n,k;\alpha, m,r) .$ This implies that, \[d^{\lambda,x}_{m,r}(n;\alpha)=\sum_{k=0}^{\infty} (x\lambda)^k\sum_{s=0}^{k}\binom{k}{s}(-1)^{k-s}(ms+r|\alpha)_n \cdot \frac{1}{m^k k!}  ,\] we obtain the falling factorial as follows, \[\frac{d^n}{dz^n}\biggl(e^{ms+r}_{\alpha}(z)\biggr)\big{|}_{{z=0}} = (ms+r|\alpha)_n ,\] \[d^{\lambda,x}_{m,r}(n;\alpha)=\sum_{k=0}^{\infty} (x\lambda)^k \frac{1}{m^k k!}\sum_{s=0}^{k}\binom{k}{s} [e^{ms+r}_{\alpha}(z)] (-1)^{k-s} ,\] 
		\[d^{\lambda,x}_{m,r}(n;\alpha)=e_{\alpha}^r (z)\sum_{k=0}^{\infty} (x\lambda)^k \frac{1}{m^k k!}\sum_{s=0}^{k}\binom{k}{s} [e^{ms}_{\alpha}(z)] (-1)^{k-s} ,\] by Binomial theorem, \[(e^{m}_{\alpha}(z)-1)^k =\sum_{s=0}^{k}\binom{k}{s} [e^{m}_{\alpha}(z)]^s (-1)^{k-s} ,\]  \[d^{\lambda,x}_{m,r}(n;\alpha)=e^{r}_{\alpha}(z)\cdot \sum_{k=0}^{\infty} (x\lambda)^k (e^{m}_{\alpha}(z)-1)^k \cdot \frac{1}{k!}\frac{1}{m^k} ,\] 	\[d^{\lambda,x}_{m,r}(n;\alpha)= e^{r}_{\alpha}(z)\cdot e^{{x\lambda}\begin{bmatrix}\frac{(1+\alpha z)^\frac{m}{\alpha}-1}{m}
		\end{bmatrix}}\]  
	\end{proof}
\end{theorem} 

	\begin{theorem}\label{Deg. r-Dowling th} For $n\geq 0$ and $\alpha, x, m,r  \in \mathbb{Z^{+}}$ such that $\alpha|m$ and $\alpha|r$ we have,
	\begin{equation}
		D^{\lambda,x}_{m,r}(n;\alpha) = \sum_{k=0}^{n} (x\lambda)^k S(n,k;\alpha, m,r) .
	\end{equation}
\end{theorem} 
\begin{proof}
In this proof, we distribute $n$ cells into $k+r$ cells. The $k$ cells are non-distinguished, and are further divided into $m$ compartments, a special cell has $r$ compartments. The $k$ cells may be colored with any any of the $x$ colors in $x^k$ ways. The cells may also be sectioned among $\lambda$ sections in ${\lambda}^k$ ways. The distribution of $n$ balls into the $k$ cells will see the first ball to be distributed will go to the favored compartment, this is controlled by $\alpha$. Also, some of the $n$ balls will go to the $r$ labeled special cell. This distribution can be done in $S(n,k;\alpha, m,r)$ ways. We then sum all the possibilities to complete the proof.  
\end{proof}

	\begin{theorem}For $n\geq 0$ and $\alpha, x, m,r  \in \mathbb{Z^{+}}$ such that $\alpha|m$ and $\alpha|r$ we have,
	\begin{equation}
		D^{\lambda, x}_{m,r}(n+1;\alpha)= rD^{\lambda, x}_{m,r-\alpha}(n;\alpha) +x\lambda D^{\lambda,x}_{m+r-\alpha,r}(n;\alpha).
	\end{equation}
\end{theorem}

\begin{proof}
The proof will be based on the placement of $(n+1)^{\text{th}}$ ball, thus we consider two cases based on whether the $(n+1)^{\text{th}}$ ball is in a special cell or an $m$ labeled cell.\\
		
\textbf{Case 1:} If the $(n+1)^{\text{th}}$ ball is in a special cell, it can occupy any of the $r$ labels in $r$ ways. This will force the next $\alpha$ compartments to close leaving $r- \alpha$ compartments free. All remaining $n$ balls will go and occupy the remaining $r-\alpha$ compartments and the $m$ labeled cells in $D^{\lambda, x}_{m,r-\alpha}(n;\alpha)$ ways.\\

\textbf{Case 2:} If the $(n+1)^{\text{th}}$ ball is in an $m$ labeled cell, the ball will first be distributed to the favored compartment, and will close the next $\alpha$ compartments leaving $m-\alpha$ compartments free. This cell can be colored with one of $x$ colors in $x$ ways and may be sectioned in $\lambda$ sections in $\lambda$ ways. Next, we form a unit of special cell and an $m$ labeled cell that contains the $(n+1)^{\text{th}}$ ball. The unit will have $m+r-\alpha$ compartments. The remaining $n$ balls will be distributed into the unit in $D^{\lambda,x}_{m+r-\alpha,r}(n;\alpha)$ ways. This completes the proof.
\end{proof}

	\begin{theorem}For $n\geq 0$ and $\alpha, x, m,r  \in \mathbb{Z^{+}}$ such that $\alpha|m$ and $\alpha|r$ we have,
	\begin{equation}
		D^{x,\lambda}_{m,r}(n+1;\alpha) = rD^{x,\lambda}_{m,r-\alpha}(n;\alpha) + x\lambda \sum_{i=0}^{n} \binom{n}{i} (m+r|\alpha)_i  B_{n-i} (\alpha,m,0, x) .
	\end{equation}
\end{theorem}
 \begin{proof}
 	We base our proof on the placement of  $(n+1)^{\text{th}}$ ball. We consider two cases.\\ 
 	
 	\textbf{Case 1:} If $(n+1)^{\text{th}}$ ball occupies special cell, it will have $r$ choices for occupation of any of $r$ compartments. This will make the next $\alpha$ compartment to close leaving $r-\alpha$ compartments free. Next, we distribute all the remaining balls into both the special cell and $m$ labeled cells in $D^{x,\lambda}_{m,r-\alpha}(n;\alpha)$ ways. \\
 	
 	\textbf{Case 2:} If the $(n+1)^{\text{th}}$ ball occupies an $m$ labeled cell, the cell will have $\alpha$ compartments closing. This cell may be colored with any of the $x$ colors in $x$ ways. It will also be sectioned in $\lambda$ sections in $\lambda$ ways. Thereafter, we combine the cell containing the $(n+1)^{\text{th}}$ ball with the special cell, this will result into a unit with $m+r-\alpha$ available compartments. To add other balls in this unit, we select some $i$ balls from the total of $n$ remaining balls in $\binom{n}{i}$ ways. The $i$ balls will be distributed in this unit in $(m+r|\alpha)_i$ ways. Remaining $n-i$ balls will then be distributed in the unit and other cells in $B_{n-i} (\alpha,m,0, x)$ ways. We then sum up all the possibilities and complete the proof.
 \end{proof}
%In LaTeX, the environment eqnarray or eqnarray* is an old configured command. If we use this old environment, the spaces before and after the suitable operators, e.g. ``+'' or ``='', will be larger than the normal case. Please don't use this old environment. Replace it with other stable and well defined other mathematical environments to handle formulas with multiple rows. For example, please use the environments align (align*) or aligned (aligned*) or multline (multline*) or gather (gathered*) etc.

\begin{theorem}For $n\geq 0$ and $\alpha, x, m,r  \in \mathbb{Z^{+}}$ such that $\alpha|m$ and $\alpha|r$ we have,
	\begin{equation}
		D^{\lambda,x}_{m,r}(n;\alpha) =\sum_{i=0}^{n} \binom{n}{i}  {\lambda}^{i}B_{i}(\alpha, m, 0, x) (r|\alpha)_{n-k} .
	\end{equation}
\end{theorem}
\begin{proof}
	To enumerate $D^{\lambda,x}_{m,r}(n;\alpha)$, we select some $i$ balls from $n$ balls in $\binom{n}{i}$ ways. The selected $i$ balls can be sectioned in one of $\lambda$ sections in ${\lambda}^{i}$ ways. The $i$ balls will go to an ordinary cell with $m$ labels, the cell may be colored with any of $x$ colors, with the first cell to be distributed going to the favored compartment. The $i$ balls will be counted by $B_{i}(\alpha, m, 0, x)$ ways. Remaining $n-i$ elements will go the special cell. After a cell is placed, next $\alpha$ compartments closes. This placement can be done in $(r|\alpha)_{n-i}$ ways.  This completes the proof.
\end{proof}
\section{Asymptotics}\label{section*}

In this section,  we provide an asymptotic expansion for higher order degenerate $r$-Dolwing polynomials, using a method proposed in \cite{hsu1990power,hsu1991kind}. The method has also been used in \cite{adell2023unified,corcino2022second,benyi2024combinatorial,Nkonkobe et al}.  
Suppose $\Theta(z)=\sum^{\infty}\limits_{n=0}m_nz^n$ is formal power series over $\mathbb{C}$ where $m_0=\Theta(0)=1$. For  $e (0 \leq e \leq n)$, let $W(n,e)$ be equal to the sum, 
\begin{equation}
	\label{result1}
	W(n,e)=\sum_{1^{g_1}2^{g_2}\cdots n^{g_n} \in \sigma (n,n-e)} \frac{m_1^{g_1}m_2^{g_2}\cdots m_n^{g_n}}{g_1!g_2! \cdots g_n!},
\end{equation}
where $\sigma (n,n-e)$ is the set of partitions of $n$ having $(n-e)$ parts. We use the following result \cite{hsu1990power,hsu1991kind}. 
Suppose $[z^n](\Theta (z))^\lambda$ is the coefficient of $z^n$ in the power series expansion of $(\Theta(z))^\lambda$. It follows that for fixed positive integer $r$ and for large $\lambda$ and $n$ such that $n=o(\lambda^{\frac{1}{2}})\: where\: \lambda \to \infty$, the following asymptotic formula holds, 
\begin{equation}
	\label{result2}
	\frac{1}{(\lambda)_n} [z^n](\lambda(z))^\lambda=\sum^{r}_{e=0} \frac{W(n,e)}{(\lambda-n+e)_e} + o\left(\frac{W(n,r)}{(\lambda-n+r)_r} \right)
\end{equation}
where $(\lambda)_e=\lambda(\lambda-1)\cdots(\lambda-(e-1))$.  
\smallskip\\
To derive the asymptotic formula, consider the generating function in \eqref{Degenerate r-Dowling generating function equation},  
\begin{equation}
	\sum_{n=0}^{\infty}D^{\lambda, x}_{m,r}(n;\alpha)\frac{z^n}{n!} = 	e_\alpha^{r}(z)e^{x\lambda[\frac{e^m_\alpha(z)-1}{m}]}.
\end{equation}
Suppose
\begin{equation}
	\Theta(t)=	\sum_{n=0}^{\infty}D^{1, x}_{m,r}(n;\alpha)\frac{z^n}{n!}=e_\alpha^{r}(z)e^{x[\frac{e^m_\alpha(z)-1}{m}]}. 
\end{equation}
It follows that
\begin{equation}
	(\Theta(t))^\lambda= e_\alpha^{r\lambda}(z)e^{x\lambda[\frac{e^m_\alpha(z)-1}{m}]}. 
\end{equation}
Using \eqref{result2} we have 
\begin{equation}\label{result3}
	\frac{D^{\lambda, x}_{m,r}(n;\alpha)}{(\lambda)_n(n!)}=\sum^{r}_{e=0} \frac{W(n,e)}{(\lambda-n+e)_e} + o\left(\frac{W(n,r)}{(\lambda-n+r)_r} \right),
\end{equation}
where $n=o(\lambda^{\frac{1}{2}})$ as $\lambda \to \infty$. We now compute few values of  $W(n,e)$,
which can be written,

$	W(n,0)=\frac{1}{n!}\begin{bmatrix}\frac{D^{\lambda, x}_{m,r}(1;\alpha)}{1!}\end{bmatrix}^{n},$ $W(n,1)=\frac{1}{(n-2)!}\begin{bmatrix}\frac{D^{\lambda, x}_{m,r}(1;\alpha)}{1!}\end{bmatrix}^{n-2}\begin{bmatrix}\frac{D^{\lambda, x}_{m,r}(2;\alpha)}{2!}\end{bmatrix},$
\begin{align*}
		W(n,2)=&\frac{1}{(n-3)!}\begin{bmatrix}\frac{D^{\lambda, x}_{m,r}(1;\alpha)}{1!}\end{bmatrix}^{n-3}\begin{bmatrix}\frac{D^{\lambda, x}_{m,r}(3;\alpha)}{3!}\end{bmatrix}+\frac{1}{2!(n-4)!}\begin{bmatrix}\frac{D^{\lambda, x}_{m,r}(1;\alpha)}{1!}\end{bmatrix}^{n-4}\begin{bmatrix}\frac{D^{\lambda, x}_{m,r}(2;\alpha)}{2!}\end{bmatrix}^2,\\
	W(n,3)=&\frac{1}{(n-4)!}\begin{bmatrix}
		\frac{D^{\lambda, x}_{m,r}(1;\alpha)}{1!}
	\end{bmatrix}^{n-4}\begin{bmatrix}
		\frac{D^{\lambda, x}_{m,r}(4;\alpha)}{4!}
	\end{bmatrix}+\frac{1}{(n-5)!}\begin{bmatrix}
		\frac{D^{\lambda, x}_{m,r}(1;\alpha)}{1!}
	\end{bmatrix}^{n-5}\begin{bmatrix}
		\frac{D^{\lambda, x}_{m,r}(2;\alpha)}{2!}
	\end{bmatrix}\begin{bmatrix}
		\frac{D^{\lambda, x}_{m,r}(3;\alpha)}{3!}
	\end{bmatrix}\\&\hspace*{2mm}+\frac{1}{3!(n-6)!}\begin{bmatrix}
		\frac{D^{\lambda, x}_{m,r}(1;\alpha)}{1!}
	\end{bmatrix}^{n-6}\begin{bmatrix}
		\frac{D^{\lambda, x}_{m,r}(5;\alpha)}{2!}
	\end{bmatrix}^3,
\end{align*}

\begin{align*}W(n,4)=&\frac{1}{(n-5)!}\begin{bmatrix}
		\frac{D^{\lambda, x}_{m,r}(1;\alpha)}{1!}
	\end{bmatrix}^{n-5}\begin{bmatrix}
		\frac{D^{\lambda, x}_{m,r}(5;\alpha)}{5!}
	\end{bmatrix}+\frac{1}{2!(n-6)!}\begin{bmatrix}
		\frac{D^{\lambda, x}_{m,r}(1;\alpha)}{1!}
	\end{bmatrix}^{n-6}\begin{bmatrix}
		\frac{D^{\lambda, x}_{m,r}(3;\alpha)}{3!}
	\end{bmatrix}^2\\&\hspace*{2mm}+\frac{1}{2!(n-7)!}\begin{bmatrix}
		\frac{D^{\lambda, x}_{m,r}(1;\alpha)}{1!}
	\end{bmatrix}^{n-7}\begin{bmatrix}
		\frac{D^{\lambda, x}_{m,r}(2;\alpha)}{2!}
	\end{bmatrix}^2\begin{bmatrix}
		\frac{D^{\lambda, x}_{m,r}(3;\alpha)}{3!}
	\end{bmatrix}\\&\hspace*{4mm}+\frac{1}{4!(n-8)!}\begin{bmatrix}
		\frac{D^{\lambda, x}_{m,r}(1;\alpha)}{1!}
	\end{bmatrix}^{n-8}\begin{bmatrix}
		\frac{D^{\lambda, x}_{m,r}(2;\alpha)}{2!}
	\end{bmatrix}^4\\&\hspace*{6mm}+
	\frac{1}{2!(n-6)!}\begin{bmatrix}
		\frac{D^{\lambda, x}_{m,r}(1;\alpha)}{1!}
	\end{bmatrix}^{n-6}\begin{bmatrix}
		\frac{D^{\lambda, x}_{m,r}(2;\alpha)}{2!}
	\end{bmatrix}\begin{bmatrix}
		\frac{D^{\lambda, x}_{m,r}(4;\alpha)}{4!}
	\end{bmatrix},
\end{align*}

Thus, the polynomials $D^{\lambda,x}_{m,r}(n;\alpha)$ is approximately the following
\begin{align*}\frac{D^{\lambda,x}_{m,r}(n;\alpha) }{n!}\sim\:  (\lambda)_nW(n,0)+(\lambda)_{n-1}W(n,1)+(\lambda)_{n-2}W(n,2)+(\lambda)_{n-3}W(n,3)+(\lambda)_{n-4}W(n,4).\end{align*}

\section{Conclusion}
In this paper, by making use of the generalized stirling numbers $S(n,k,\alpha, \beta, \gamma)$ we defined a new higher order degenerate generalization of the Dowling polynomials. We also provided combinatorial interpretations of some identities of this generalization. We discussed both their degenerate and their non-degenerate versions. We also provided some of their integral representations.  Finally we provided some of their asymptotic results.

For further works a variation of our combinatorial definition of the polynomials $B_n(x,\alpha, \beta,r)$ discussed in this study may be done in the following ways. By requiring each cell to have unlimited capacity, while also not requiring each time an element lands on a compartment $\alpha$ compartments close. This would be a non-degenerate version of the polynomials $B_n(x,\alpha, \beta,r)$. Another variation could be the following. As defined in \cite{corcino2001} the generalized stirling numbers $(-1)^{n+k}S(n,k,\alpha,\beta,\gamma)$ give number of partitions where one of the conditions is that each time an element lands on a compartment, the compartment splits into $\alpha +1$ compartments, a variation of our combinatorial model for the polynomials $B_n(x,\alpha, \beta,r)$ may be done to include this condition instead. One extension of our work could be similar to the associated Stirling numbers (see \cite{howard1980associated}), one may define associated higher order-$r$ Dowling polynomials, by restricting  minimal number of elements each subset should contain. Without going into higher order version discussed in this in the article \cite{gyimesi2021associated} the authors discussed  associated $r$- Dowling numbers, this may be extended to the higher order $r$-Dowling polynomials discussed in this study. Also probabilistic higher order $r$-Dowling polynomials may be defined. Without going into higher order version discussed in this in the article \cite{kim2025probabilistic} probabilistic degenerate Dowling polynomials associated with random variables, their results may be extended to the higher order $r$-Dowling polynomials we define in this study.

% of the  higher order $r$-Dowling polynomials including their combinatorial interpretation and integral representations. In section \ref{section*} we explore asymptotics results for the higher order $r$-dowling polynomials. 

\end{document}